\documentclass[a4paper,6pt]{article}
\usepackage[utf8]{inputenc}
   \usepackage{hyperref}
   \usepackage{amsmath}
   \usepackage{amsfonts}
   \usepackage{amssymb}
   \usepackage{mathrsfs}
   \usepackage{graphicx}
   \usepackage[abs]{overpic}
   \usepackage{float}
   \usepackage{cases}
\newtheorem{thm}{Theora}[section]
\newtheorem{Theo}[thm]{Theorem}

\newtheorem{Cor}[thm]{Corollary}
\newtheorem{Lem}[thm]{Lemma}
\newtheorem{Prop}[thm]{Proposition}

\newenvironment{theorem*}[1]{\smallskip\noindent{\bf #1.}\rm}{\medskip}

\newenvironment{Proof}{\smallskip\noindent{\bf Proof.}\rm}
{\hfill $\Box$\medskip}
\newenvironment{proof}{\smallskip\noindent{\bf Proof}\rm}
{\hfill $\Box$\medskip}
\renewcommand\({\left(}
\renewcommand\){\right)}

\newcommand\la{\lambda}

\newcommand{\be}{\begin{equation}}
\newcommand{\ee}{\end{equation}}
\newcommand{\ba}{\begin{array}}
\newcommand{\ea}{\end{array}}
\newcommand{\bea}{\begin{eqnarray*}}
\newcommand{\eea}{\end{eqnarray*}}
\newcommand{\bean}{\begin{eqnarray}}
\newcommand{\eean}{\end{eqnarray}}
\newcommand\D{{Dim}}

\makeatletter \@addtoreset{equation}{section}

\makeatother
\begin{document}
\title{ Null Boundary Controllability Of A One-dimensional Heat Equation With An Internal
 Point Mass And Variable Coefficients}
\author{Jamel Ben Amara \thanks{Facult\'e des Sciences de Tunis, D\'{e}partement de Math\'{e}matiques, Laboratoire d'Ing\'{e}nierie Math\'{e}matique,
Ecole Polytechnique de Tunisie, Universit\'{e} de Carthage, Avenue
de la R\'{e}publique, Bp 77, 1054 Amilcar, Tunisia,
jamel.benamara@fsb.rnu.tn.}~~~~~~~~~~Hedi Bouzidi \thanks{Facult\'e
des Sciences de Tunis, D\'{e}partement de Math\'{e}matiques and
Laboratoire d'Ing\'{e}nierie Math\'{e}matique, Ecole Polytechnique
de Tunisie~, Tunisia, bouzidihedi@yahoo.fr.}}
\date{}
\maketitle {\bf Abstract:} In this paper we consider a linear hybrid
system which composed by two non-homogeneous rods connected by a
point mass and generated by the equations\bea\left\{
 \begin{array}{ll}
   \rho_{1}(x)u_{t}=(\sigma_{1}(x)u_{x})_{x}-q_{1}(x)u,& x\in(-1,0),~t>0, \\
   \rho_{2}(x)v_{t}=(\sigma_{2}(x)v_{x})_{x}-q_{2}(x)v,& x\in(0,1), ~~~t>0, \\
   u(0,t)=v(0,t)=z(t),&t>0,\\
   M z_{t}(t)=\sigma_{2}(0)v_{x}(0,t)-\sigma_{1}(0)u_{x}(0,t),&t>0,
\end{array}
 \right.
\eea with Dirichlet boundary condition on the left end $x=-1$ and a
boundary control acts on the right end $x=1$. We prove that this
system is null controllable with Dirichlet or Neumann boundary
controls. Our approach is mainly based on a detailed spectral
analysis together with the moment method. In particular, we show
that the associated spectral gap in both cases (Dirichlet or Neumann
boundary controls) are positive without further conditions on the
coefficients $\rho_{i}$, $\sigma_{i}$ and $q_{i}$ $(i=1,2)$ other
than the regularities.\\

{\bf Keywords.} Heat equation, nonhomogeneous, point masses,
boundary control, moments.\\

{\bf AMS subject classification.} 35K05, 93B05, 93B55, 93B60.
\section{Introduction}
The null controllability of parabolic equations has been extensively
investigated for several decades. After the pioneering works by D.
Russell and H. Fattorini \cite{DH, DH1}, there is a significant
progress in the $N$-dimensional case by using Carleman estimates, see
in particular \cite{AO}. The more recent developments of the theory
are concerned with degenerate coefficients \cite{PPJ1, PPJ},
discontinuous coefficients \cite{GL, AYJ, EE}, or singular
coefficients see in particular \cite{JE}.\\
In this paper we consider a one-dimensional linear hybrid system
which composed by two non-homogeneous rods connected at $x=0$ by a
point mass. We assume that the first rod occupies the interval
$(-1,0)$ and the second one occupies the interval $(0,1)$. The
temperature of the first and the second rod will be respectively
presented by the functions \bea
u&=&u(x,t),~~x\in (-1,0),~~t>0,\\
v&=&v(x,t),~~x\in (0,1),~~~~t>0. \eea The position of the mass $M >
0$ attached to the rods at the point $x = 0$ is denoted by the
function $z = z(t)$ for $t > 0$. The equations modeling the dynamic
of this system are the followings \bean\label{a} \left\{
 \begin{array}{ll}
   \rho_{1}(x)u_{t}=(\sigma_{1}(x)u_{x})_{x}-q_{1}(x)u,& x\in(-1,0),~t>0, \\
   \rho_{2}(x)v_{t}=(\sigma_{2}(x)v_{x})_{x}-q_{2}(x)v,& x\in(0,1),~~~~t>0, \\
   u(0,t)=v(0,t)=z(t),& t>0,\\
   M z_{t}(t)=\sigma_{2}(0)v_{x}(0,t)-\sigma_{1}(0)u_{x}(0,t), &t>0, \\
u(-1,t)=0,
\end{array}
 \right.
\eean with either Dirichlet boundary control
\be\label{b}v(1,t)=h(t),~t>0,\ee or Neumann boundary control
\be\label{c} v_{x}(1,t)=h(t),~~t>0.\ee In System \eqref{a} the first
two equations are the one-dimensional heat equation. The third
equation imposes the continuity of the three components of the
system at $x =0$. The fourth equation describes the change in
temperature of the point mass at $x=0$. The coefficients
$\rho_{i}(x)$ and $\sigma_{i}(x)$ $(i=1,2)$ represent respectively
the density and thermal conductivity of each rod. The potentials are
assumed positively and denoted by the functions $q_{1}(x)$ and
$q_{2}(x)$. Similar hybrid systems involving strings and beams with
point masses have been studied in the context of controllability
(see e.g.,\cite{CAS, Beam1,E.Z}). Throughout this paper, we assume
that the coefficients $\rho_{i}$, $\sigma_{i}$ and $q_{i}$ $(i=1,2)$
are uniformly positive such that
\bean&\rho_{1},~\sigma_{1} \in H^{2}(-1,0),~q_{1}\in C(-1,0),\label{i1}\\
&\rho_{2},~\sigma_{2} \in H^{2}(0,1),~q_{2}\in
C(0,1).\label{i2}\eean In order to determine the solution of
Systems \eqref{a}-\eqref{b} and \eqref{a}-\eqref{c} in an unique
way, we have to add some initial conditions at time $t=0$ that will
be represented by \be\label{d} \left\{
\begin{array}{l}
  u(x,0)=u^{0}(x),~~x\in(-1,0), \\
  v(x,0)=v^{0}(x),~~x\in(0,1), \\
  z(0)=z^{0},\\
\end{array}
 \right.
\ee where the triple $\{u^{0}, v^{0},z^{0}\}$ will be given in an
appropriately defined function space. According to the results in
\cite{HM1, J.M}, the solutions of the System \eqref{a} with the
homogenous Dirichlet boundary condition \be v(1,t)=0,~~t>0,
\label{w1}\ee can be regarded as weak limits of solutions of a heat
equations with densities $\rho_{1}(x)$, $\rho_{2}(x)$ on the
intervals $(-1,-\epsilon)$ and $(\epsilon,1)$, respectively
and with the density $\frac{1}{2\epsilon}$ on the interval $(-\epsilon,\epsilon)$.\\
Note that when $M=0$, we recover the continuity condition of $u_{x}$
at $x=0$ and the classical heat equation with variable coefficients
occupying the interval $(-1,1)$ without point mass. In this context,
the question of the null controllability of Problems
\eqref{a}-\eqref{b} and \eqref{a}-\eqref{c} (for $M=0$) have been
treated in the seminal papers \cite{DH,DH1} and also \cite{EE,AE}
for $q=0$, $\sigma=1$ and some additional conditions on the density
$\rho(x)$. Recently, Hansen and Martinez \cite{HM} studied the
boundary controllability of Systems \eqref{a}-\eqref{b} and
\eqref{a}-\eqref{c} in the case of constant coefficients
$\rho_{i}(x)\equiv\sigma_{i}(x)\equiv1$, $q_{i}(x)\equiv
0,~~(i=1,2)$ and $M=1$. They proved the null boundary
controllability of Problems \eqref{a}-\eqref{b} and
\eqref{a}-\eqref{c} by using the moment
method.\\
In this paper we prove the null controllability of Systems
\eqref{a}-\eqref{b} and \eqref{a}-\eqref{c} at any time $T>0$. Our
approach is essentially based on a precise computation of the
associated spectral gap together with the moment method. More
precisely, we show that the sequence of eigenvalues
$(\lambda_{n})_{n\in\mathbb{N}^{*}}$ and
$(\nu_{n})_{n\in\mathbb{N}^{*}}$ associated with systems
\eqref{a}-\eqref{b} and \eqref{a}-\eqref{c}, respectively, satisfy
the gap conditions $\la_{n+1}-\lambda_{n}\geq\delta_{1}>0 $,
$\nu_{n+1}-\nu_{n}\geq\delta_{2}>0,~n\geq1$, without further
conditions on $\rho_{i}$, $\sigma_{i}$ and $q_{i}$, $(i=1,2)$ other
than the regularities. In the process of the computation of the
spectral gap, we establish an interpolation formula between the
eigenvalues of System \eqref{a}-\eqref{b} and the eigenvalues of the
regular problem \eqref{a}-\eqref{b} for $M=0$. We think that this
result can be useful for other problems related to Systems
\eqref{a}-\eqref{b} and
\eqref{a}-\eqref{c} without controls.\\
The rest of the paper is divided in the following way: In section 2
we associate to System \eqref{a}-\eqref{w1} a self-adjoint operator
defined in a well chosen Hilbert space. Moreover, we give some
results concerning the well-posedness of system \eqref{a}
with the homogenous Dirichlet boundary condition \eqref{w1}.
In the next section we establish the asymptotic properties of the associated
spectral gap and the asymptotes of the eigenfunctions. In Section 4
we reduce the control problem \eqref{a}-\eqref{b} to a moment problem and we prove the null boundary
controllability of System \eqref{a}-\eqref{b}. Finally, in Section 5
we extend our results to the case of Neumann boundary control
\eqref{a}-\eqref{c}.
\section{ Operator Framework And Well-posedness }\label{s1}
In this section we investigate the well-posedness of System
(\ref{a}) with the homogeneous Dirichlet boundary
condition \eqref{w1}. 
In order to state the main results of this section, it is convenient
to introduce the following spaces \bean
&&\mathcal{V}_{1}=\{u\in H^{1}(-1,0)~~|~~ u(-1)=0\}\label{w3}\\
&&\mathcal{V}_{2}=\{v\in H^{1}(0,1)~~|~~ v(1)=0\}\nonumber\\
&&\mathcal{V}=\{(u,v)\in \mathcal{V}_{1}\times
\mathcal{V}_{2}~~|~~u(0)=v(0)\}\nonumber\eean endowed with the norms
\bea&&\|u\|^{2}_{\mathcal{V}_{1}}=\int_{-1}^{0}|u_{x}(x)|^{2}dx,\\
&&\|v\|^{2}_{\mathcal{V}_{2}}=\int_{0}^{1}|v_{x}(x)|^{2}dx,\\
&&\|(u,v)\|^{2}_{\mathcal{V}}=\|u\|^{2}_{\mathcal{V}_{1}}+\|v\|^{2}_{\mathcal{V}_{2}}.\eea
Let us consider the following closed subspace of $\mathcal{V} \times
\mathbb{R}$
\begin{equation}\label{v}
\mathcal{W}=\{(u,v,z)\in \mathcal{V} \times \mathbb{R}:
u(0)=v(0)=z\},
\end{equation}
equipped with norm
$\|(u,v,z)\|^{2}_{\mathcal{W}}=\|(u,v)\|^{2}_{\mathcal{V}}$.
Let us define the Hilbert space
\begin{equation}\label{h}
\mathcal{H}=L^{2}(-1,0)\times L^{2}(0,1)\times \mathbb{R},
\end{equation}
with the scalar product
$\langle./.\rangle_{\mathcal{H}}$ defined by\\
for all $Y_{1}=(u_{1},v_{1},\alpha_{1})^{t}$ and
$Y_{2}=(u_{2},v_{2},\alpha_{2})^{t}\in \mathcal{H}$, where $^{t}$
denotes the transposition, we have
\begin{equation}\label{dv}
\langle Y_{1},
Y_{2}\rangle_{\mathcal{H}}=\int_{-1}^{0}u_{1}(x)u_{2}(x)\rho_{1}(x)dx+
\int_{0}^{1}v_{1}(x)v_{2}(x)\rho_{2}(x)dx+M\alpha_{1}\alpha_{2}.
\end{equation}
It is easy to show that the space $\mathcal{W}$ is densely and
continuously embedded in the space $\mathcal{H}$. In the sequel we
introduce the operator $\mathcal{A}$ defined in $\mathcal{H}$ by
setting
\begin{equation}\label{est}
\mathcal{A}Y=\begin{cases} \frac{1}{\rho_{1}}(-(\sigma_{1}u')'+q_{1}u),~~~~&x\in[-1,0],\\
\frac{1}{\rho_{2}}(-(\sigma_{2}v')'+q_{2}v),~~~&x\in[0,1],\\
\frac{1}{M}(\sigma_{1}(0)u'(0)-\sigma_{2}(0)v'(0)),~~&x=0,
\end{cases}
\end{equation}
where $Y=(u,v,z)^{t}$. The domain $D(\mathcal{A})$ of $\mathcal{A}$
is dense in $\mathcal{H}$ and is given by
$$
D(\mathcal{A})=\{Y=(u, v, z)\in \mathcal{W} ~:~ (u,v)\in
H^{2}(-1,0)\times H^{2}(0,1)\}.
$$
\begin{Prop}\label{rr}
The linear operator $\mathcal{A}$ is a positive and a self-adjoint such that $\mathcal{A}^{-1}$ is compact. Moreover,
$\mathcal A$ is the infinitesimal generator of a strongly continuous
semigroup $(\mathbf{S}_{t})_{t\geq0}$.
\end{Prop}
\begin{Proof}
Let $Y=(u,v,z)^{t}\in D(\mathcal{A})$, then by a simple integration by parts we have\\
\bean \langle \mathcal{A}Y,
Y\rangle_{\mathcal{H}}&=&\int_{-1}^{0}(-(\sigma_{1}(x)u^{'})^{'}+q_{1}(x)u)u
dx
+\int_{0}^{1}(-(\sigma_{2}(x)v^{'})^{'}+q_{2}(x)v)v dx+Mz_{t}z\nonumber\\
&=&\int_{-1}^{0}\sigma_{1}(x)|u^{'}|^{2}dx+q_{1}(x)|u|^{2}dx
+\int_{0}^{1}\sigma_{2}(x)|v^{'}|^{2}dx+q_{2}(x)|v|^{2}dx\label{w0}
\eean
 since $\sigma_{i}>0$ and $q_{i}>0$ (i=1,2) then $\langle
\mathcal{A}Y, Y\rangle_{\mathcal{H}}>0.$\\
It is clear that the quadratic form has real values so the linear
operator $\mathcal{A}$ is symmetric. In order to show that this
operator is self-adjoint it suffices to show that
$Ran(\mathcal{A}-iId)=\mathcal{H}$.\\
It is easy to show that the space $\mathcal{W}$ is continuously and
compactly embedded in the space $\mathcal{H}$, and hence, the
operator $\mathcal{A}^{-1}$ is compact in $\mathcal{H}$.
\end{Proof}\\
We consider the following spectral problem which arises by applying
separation of variables to System (\ref{a})-(\ref{w1})
\bean\label{M} \left\{
 \begin{array}{ll}
-(\sigma_{1}(x)u')'+q_{1}(x)u=\lambda \rho_{1}(x)u,~~&x\in (-1,0), \\
-(\sigma_{2}(x)v')'+q_{2}(x)v=\lambda \rho_{2}(x)v,~~&x\in(0,1), \\
u(-1)=v(1)=0,~u(0)=v(0), \\
\sigma_{1}(0)u'(0)-\sigma_{2}(0)v'(0)=\lambda M u(0).\\
\end{array}
 \right.
\eean
\begin{Lem}\label{sss}
The spectrum of System \eqref{M} is discrete. It consists of an
increasing sequence of positive and simple
eigenvalues $(\lambda_{n})_{n\in\mathbb{N}^{*}}$ tending to $+\infty$\\
$$0<\lambda_{1}<\lambda_{2}<.......<\lambda_{n}<.....\underset{n\rightarrow \infty}{\longrightarrow}\infty .$$
Moreover, the corresponding eigenfunctions
$(\widetilde{\phi}_{n}(x))_{n\in\mathbb{N^{*}}}$ form an orthonormal basis in
$\mathcal{H}$.
\end{Lem}
\begin{Proof}
Here we have only to prove the simplicity of the eigenvalues
$\la_{n}$ for all $n\in\mathbb{N^{*}}$. Let $u(x,\la)$ and
$v(x,\la)$ be the solutions of the initial value problems
\begin{equation}\label{g6}
\left\{
\begin{array}{lll}
-(\sigma_{1}(x)u')'+q_{1}(x)u=\lambda \rho_{1}(x) u, ~~x\in (-1,\,0),\\
u(-1)=0,u'(-1)=1,
\end{array}
\right.
\end{equation} and
 \begin{equation}\label{g7}
\left\{
\begin{array}{lll}
-(\sigma_{2}(x)v')'+q_{2}(x)v=\lambda \rho_{2}(x)v,~~x\in (0,\,1), \\
v(1)=0,v'(1)=-1,
\end{array}
\right.
\end{equation}
respectively. Let $\lambda$ be an eigenvalue of the operator
$\mathcal{A}$ and $E_{\lambda}$ be the corresponding eigenspace. For
any eigenfunction $\phi (x,\lambda)$ of $E_{\lambda}$, $\phi
(x,\lambda)$ can be written in the form
\begin{equation} \phi(x,\la)= \left\{
\begin{array}{ll}
c_{1}u(x,\la),&-1\leq x\leq 0,\\
c_{2}v(x,\la),&0\leq x\leq 1,\\
\end{array}
\right.
\end{equation}
where $c_{1}$ and $c_{2}$ are two constants. The second condition at
$x=0$ in \eqref{M} is equivalent to
$$c_{1}u(0,\la)=c_{2}v(0,\la).$$
If $u(0,\la)\neq 0$ and $v(0,\la)\neq 0$, then $c_{1}=c_{2}
\frac{v(0,\la)}{u(0,\la)}$
and $\D(E_{\lambda})=1$.\\
If $u(0,\la)=0$ and $v(0,\la)\neq0$ (or $u(0,\la)\neq0$ and $v(0,\la)=0$), then
\begin{equation*} \phi(x,\la)= \left\{
\begin{array}{lll}
c_{1}u(x,\la),~&-1\leq x\leq 0,\\
0 ~&0\leq x\leq 1.\\
\end{array}
\right.
\end{equation*}
It is clear from the last condition in Problem \eqref{M} that $\la$ is not an eigenvalue.\\
Now, if $u(0,\la)=v(0,\la)=0$, then from the last condition in \eqref{M} we have
$$c_{1}\sigma_{1}(0)u'(0,\la)-c_{2}\sigma_{2}(0)v'(0,\la)=0.$$
Since $u'(0,\lambda) \neq 0$ and $v'(0,\lambda) \neq 0$, then
$c_{1}=c_{2}\frac{\sigma_{2}(0)v'(0,\lambda)}{\sigma_{1}(0)u'(0,\lambda)}$ and $\D(E_{\lambda})=1$.\\
Since the operator $\mathcal{A}$ is self-adjoint then the algebraic
multiplicity of $\la$ is equal to one.
\end{Proof}\\
Obviously, the Cauchy problem (\ref{a})-(\ref{w1}) can be rewritten
in the abstract form
\begin{equation}\label{jik}
\dot{Y}(t)=-\mathcal{A}Y(t),~~Y^{0}=Y(0),~~ t>0,
\end{equation}
where $\mathcal{A}$ is defined in (\ref{est}) and
$Y^{0}=(u(x,0),v(x,0),z(0))^{t}$. As a consequence of Proposition \ref{rr}
and the Lumer-Phillips theorem (e.g., \cite[Theorem A.4]{J}),
we have the following existence and uniqueness result for the
problem (\ref{a})-(\ref{w1}):
\begin{Prop}
Let $Y^{0}=(u^{0},v^{0},z^{0})^t\in \mathcal{H}$. Then the Cauchy
problem (\ref{a})-(\ref{w1}) has a unique solution $Y=(u,v,z)^t\in
C([0,\infty), \mathcal{H})$.
\end{Prop}
\section{Spectral Gap And Asymptotic Proprieties }\label{GD1}
In this section we investigate the asymptotic behavior of the
spectral gap $\la_{n+1}-\la_{n}$ for large $n$. Namely, we enunciate
the following result:
\begin{Theo}\label{T1}
For all $n\in \mathbb{N^{*}}$, there is a constant $\delta_{1}>0$
such that the sequence of eigenvalues $(\la_{n})_{n\geq1}$ of the
spectral problem \eqref{M} satisfy the asymptotic \bean
\la_{n+1}-\la_{n}\geq\delta_{1},~n\in \mathbb{N}^{*}\label{g26}.
\eean
\end{Theo}
In order to prove this theorem, we establish some preliminary results. Let\\
$\Gamma=\{\mu_{n}\}_{1}^{\infty}=\{\eta_{j}\}_{1}^{\infty}\bigcup\{\eta'_{k}\}_{1}^{\infty}$
where $\eta_{j}$ and $\eta'_{k}$ are the eigenvalues of the
problems\begin{equation}\label{g1} \left\{
\begin{array}{lll}
-(\sigma_{1}(x)y')'+q_{1}(x)y=\lambda\rho_{1}(x)y, x\in (-1,\,0),\\
y(-1)=y(0)=0,
\end{array}
\right.
\end{equation}and
 \begin{equation}\label{g2}
\left\{
\begin{array}{lll}
-(\sigma_{2}(x)y')'+q_{2}(x)y=\lambda\rho_{2}(x)y, x\in (0,\,1),\\
y(0)=y(1)=0,
\end{array}
\right.
\end{equation}
respectively. Obviously, $\eta_{j}$ and $\eta'_{k}$ can be coincide.
Let \be\Gamma^{*}=\{\mu_{n}\in\Gamma~\backslash~
\eta_{j}=\eta'_{k},~j,~k\in\mathbb{N^{*}}\}. \label{g34}\ee
Note that if $\mu_{n}\in\Gamma^{*}$ (i.e., $u(0,\mu_n)=v(0,\mu_n)=0$), then $\mu_n$ is an eigenvalue of both
Problems \eqref{g1} and \eqref{g2}. In what follows
we suppose that  if $\mu_n\in \Gamma^{*}$, then $\mu_n=\mu_{n+1}$.   
We consider the following boundary value problem
\begin{equation}\label{g3}
\left\{
\begin{array}{lll}
-(\sigma_{1}(x)u')'+q_{1}(x)u=\lambda \rho_{1}(x) u, x\in (-1,\,0),\\
-(\sigma_{2}(x)v')'+q_{2}(x)v=\lambda \rho_{2}(x)v, x\in (0,\,1),\\
u(-1)=v(1)=0,\\
u(0)=v(0).\\
\end{array}
\right.
\end{equation}
It is clear that for $\la \in \mathbb{C}\backslash\ \Gamma^{*}$, the
set of solutions of Problem \eqref{g3} is one-dimensional subspace
which is generated by a solution of the form
\begin{equation}\label{g4}
\widetilde{U}(x,\la)= \left\{
\begin{array}{lll}
{v(0,\la)} u(x,\la),&-1\leq x\leq 0,\\
{u(0,\la)} v(x,\la),&0\leq x\leq 1,\\
\end{array}
\right.
\end{equation}
where $u(x,\la)$ and $v(x,\la)$ are the solutions of the initial
value problems \eqref{g6} and \eqref{g7}, respectively. Note that
$u(0,\la)\neq0$ and $v(0,\la)\neq0$ for $\la
\in(\mu_{n},\mu_{n+1})$, since otherwise $\la$ would be an
eigenvalue of one of Problems \eqref{g1} or \eqref{g2}. Let us
introduce  the variable complex function
$$F(\lambda)=\dfrac{\sigma_{1}(0)\widetilde{U}_{x}(0^-,\,\lambda)-
\sigma_{2}(0)\widetilde{U}_{x}(0^+,\,\lambda)}{\widetilde{U}(0,\,\lambda)},~~\la\in\mathbb{C}\backslash\Gamma,$$
which can be rewritten in the form \be
F(\lambda)=\frac{\sigma_{1}(0)v(0)u_{x}(0^-)-\sigma_{2}(0)u(0)v_{x}(0^+)}{u(0)v(0)},~~\la\in\mathbb{C}\backslash\Gamma.\label{g5}\ee
It is known in \cite[Chapter1]{BI}, $u(x,\la)$ and $v(x,\la)$ are
entire functions in $\la$ and continuous on the intervals $[-1,0]$
and $[0,1]$, respectively. Therefore $F(\la)$ is a
 meromorphic function. We will show below that its zeros and poles coincide with the
eigenvalues of the regular problem \eqref{M} (for $M=0$) and the
eigenvalues $\mu_{n}$, $n\geq1$, respectively. Moreover, the
solution of the equation \be F(\la)=M\la,\label{g14}\ee are the
eigenvalues $\la_{n}$, $n\geq 1$, of Problem \eqref{M}.
\begin{Lem}\label{Lg1}
The function $F(\lambda)$ is decreasing along the intervals
$(-\infty,\mu_{1})$ and $(\mu_{n}, \mu_{n+1})$, $n\geq1$ (with
$\mu_{n}\neq\mu_{n+1}$). Furthermore, it decreases from $+\infty$
to $-\infty$.
\end{Lem}
\begin{Proof}
Let $(\la, \la')\in (\mu_{n}, \mu_{n+1})$
where $\la\neq\la'$ and $\widetilde{U}(x,\la)$,
$\widetilde{U}(x,\la')$ are two solutions of Problem \eqref{g3}.
Integrating by parts and taking into account the boundary conditions
in \eqref{g3}, yield \bean\label{g61} \left\{
 \begin{array}{ll}
    \textstyle\int_{-1}^{0} (q_{1}-\la \rho_{1})(x)u(x,\la)u(x,\la')+\sigma_{1}(x) u_{x}(x,\la) u_{x}(x,\la')dx
     = \sigma_{1}(0)u(0,\la')u_{x}(0^-,\la), \\
     \\
   \textstyle \int_{-1}^{0} (q_{1}-\la' \rho_{1})(x)u(x,\la')u(x,\la)+\sigma_{1}(x) u_{x}(x,\la') u_{x}(x,\la)dx
    =\sigma_{1}(0) u(0,\la)u_{x}(0^-,\la'),
\end{array}
 \right.
\eean and \bean\label{g71} \left\{
 \begin{array}{ll}
    \textstyle\int_{0}^{1} (q_{2}-\la \rho_{2})(x)v(x,\la)v(x,\la')+\sigma_{2}(x) v_{x}(x,\la) v_{x}(x,\la')dx
     = -\sigma_{2}(0)v(0,\la')v_{x}(0^+,\la), \\
     \\
    \textstyle\int_{0}^{1} (q_{2}-\la' \rho_{2})(x)v(x,\la')v(x,\la)+\sigma_{2}(x) v_{x}(x,\la') v_{x}(x,\la)dx
    = -\sigma_{2}(0) v(0,\la)v_{x}(0^+,\la').
\end{array}
 \right.
\eean Subtracting the two equations of Systems \eqref{g61} and
\eqref{g71}, we obtain \bea \left\{
 \begin{array}{ll}
    (\la'-\la)\displaystyle\int_{-1}^{0} \rho_{1}(x)u(x,\la)u(x,\la')dx
     = \sigma_{1}(0) \left(u(0,\la')u_{x}(0^-,\la)-u(0,\la)u_{x}(0^-,\la')\right), \\
     (\la'-\la)\displaystyle\int_{0}^{1} \rho_{2}(x)v(x,\la)v(x,\la')dx
     = \sigma_{2}(0)\left(v(0,\la)v_{x}(0^+,\la')-v(0,\la')v_{x}(0^+,\la)\right).
\end{array}
 \right.
\eea Hence \bean (\la-\la')\displaystyle\int_{-1}^{0}
\rho_{2}(x)u(x,\la)u(x,\la')&dx&
     =\sigma_{1}(0) u(0,\la)\left(u_{x}(0^-,\la')-u_{x}(0^-,\la)\right)\nonumber\\
     &&-\sigma_{1}(0)u_{x}(0^-,\la)\left(u(0,\la')-u(0,\la)\right)\label{f1}
\eean and \bean (\la'-\la)\displaystyle\int_{0}^{1}
\rho_{2}(x)v(x,\la)v(x,\la')&dx&
= \sigma_{2}(0) v(0,\la)\left(v_{x}(0^+,\la')-v_{x}(0^+,\la)\right)\nonumber\\
     &&-\sigma_{2}(0)v_{x}(0^+,\la)\left(v(0,\la')-v(0,\la)\right).\label{f2}
\eean Passing to the limit as $\la' \rightarrow \la$ in \eqref{f1} and \eqref{f2}, we get the
identities \bean\label{g16} \left\{
 \begin{array}{ll}
-\displaystyle\int_{-1}^{0} \rho_{1}(x)u^{2}(x,\la)dx =\sigma_{1}(0)
 \left(u(0,\la)\frac{\partial u_{x}(0^-,\la)}{\partial\la}-u_{x}(0^-,\la)\frac{\partial u(0,\la)}{\partial\la}\right),\\
-\displaystyle\int_{0}^{1} \rho_{2}(x)v^{2}(x,\la)dx= -\sigma_{2}(0)
\left(v(0,\la)\frac{\partial v_{x}(0^+,\la)}{\partial
\la}-v_{x}(0^+,\la)\frac {\partial v(0,\la)}{\partial \la}\right).
\end{array}
 \right.
\eean Dividing the first equation in \eqref{g16} by $u^2(0,\la)$ and
the second by $v^2(0,\la)$, it follows \be\dfrac{\partial
F(\la)}{\partial\la}=-\dfrac{v^2(0,\la)\int_{-1}^{0}
\rho_{1}(x)u^{2}(x,\la)dx+ u^2(0,\la)\int_{0}^{1}
\rho_{2}(x)v^{2}(x,\la)dx}{u^2(0,\la)v^2(0,\la)}<0.\label{g17}\ee In
order to prove the second statement, we firstly establish the
asymptotic of $F(\la)$ as $\la\rightarrow-\infty$. It is known
(e.g., \cite[Chapter 2]{M} and \cite[Chapter 1]{BI}), for
$\la\in\mathbb{C}$ and $|\la|\rightarrow \infty$ that
\bean\label{g21}
\begin{cases}
u(x,\la)=
\left(\rho_{1}(x)\sigma_{1}(x)\right)^{-\frac{1}{4}}\dfrac{\sin\left(\sqrt{\la}\int_{-1}^{x}
\sqrt{\frac{\rho_{1}(t)}{\sigma_{1}(t)}}dt\right)}{\gamma_{1}\sqrt{|\la|}}[1],\\
 u_{x}(x,\la)=
\dfrac{\left(\rho_{1}(x)\right)^{\frac{1}{4}}\left(\sigma_{1}(x)\right)^{-\frac{3}{4}}}{\gamma_{1}}\cos\left(\sqrt{\la}\int_{-1}^{x}
\sqrt{\frac{\rho_{1}(t)}{\sigma_{1}(t)}}dt\right)[1],
\end{cases}
\eean and \bean\label{g22}
\begin{cases}
v(x,\la)=
\left(\rho_{2}(x)\sigma_{2}(x)\right)^{-\frac{1}{4}}\dfrac{\sin\left(\sqrt{\la}\int_{x}^{1}
\sqrt{\frac{\rho_{2}(t)}{\sigma_{2}(t)}}dt\right)}{\gamma_{2}\sqrt{|\la|}}[1],\\
v_{x}(x,\la)=
-\dfrac{\left(\rho_{2}(x)\right)^{\frac{1}{4}}\left(\sigma_{2}(x)\right)^{-\frac{3}{4}}}{\gamma_{2}}\cos
\left(\sqrt{\la}\int_{x}^{1}
\sqrt{\frac{\rho_{2}(t)}{\sigma_{2}(t)}}dt\right)[1],
\end{cases}
\eean where $[1]=1+\mathcal{O}(\frac{1}{\sqrt{|\la|}})$, \bean
\gamma_{1}=\int_{-1}^{0}\sqrt{\frac{\rho_{1}(x)}{\sigma_{1}(x)}}dx,~~and~~~
\gamma_{2}=\int_{0}^{1}\sqrt{\frac{\rho_{2}(x)}{\sigma_{2}(x)}}dx.
\label{g36}\eean By use of \eqref{g21} and \eqref{g22}, a
straightforward calculation gives the following asymptotic \be
\label{g15}F(\la)\sim\sqrt{|\la|}\big(\sqrt{\rho_{1}(0)\sigma_{1}(0)}+
\sqrt{\rho_{2}(0)\sigma_{2}(0)}\big),~as~\la\rightarrow-\infty. \ee
This implies that $\lim\limits_{\lambda\rightarrow
-\infty}F(\lambda)=+\infty$. Now, we prove that \be
\lim_{\la\rightarrow\mu_{n}+0}F(\la)=+\infty,~~~~
\lim_{\la\rightarrow\mu_{n}-0}F(\la)=-\infty.\label{g18}\ee If
$u(0,\mu_{n})=v(0,\mu_{n})$=0, then
$u(0,\mu_{n})v_{x}(0,\mu_{n})-u_{x}(0,\mu_{n})v(0,\mu_{n})=0$. We
put $\la=\mu_{n}+\epsilon$, where $\epsilon$ is small enough.
Therefore a simple calculation yields
  \bean F(\lambda)=&\dfrac{\sigma_{1}(0)u_{x}(0,\mu_{n})\frac{\partial v}{\partial\la}(0,\mu_{n})
-\sigma_{2}(0)v_{x}(0,\mu_{n})\frac{\partial
u}{\partial\la}(0,\mu_{n})}{\epsilon \frac{\partial
u}{\partial\la}(0,\mu_{n})\frac{\partial v}
{\partial\la}(0,\mu_{n})}+ \mathcal{O}(1)\nonumber \\
=&\frac{1}{\epsilon}\left(\dfrac{\sigma_{1}(0)u_{x}(0,\mu_{n})}{\frac{\partial
u}{\partial\la}(0,\mu_{n})}- \dfrac{\sigma_{2}(0) v_{x}(0,\mu_{n})}{
\frac{\partial v}{\partial\la}(0,\mu_{n})}\right)+\mathcal{O}(1).\label{g13}
\eean Since $\mu_{n}$ is a simple eigenvalue of the two problems
\eqref{g1} and \eqref{g2}, then $\frac{\partial
u}{\partial\la}(0,\mu_{n}) \neq 0$ and
  $\frac{\partial v}{\partial\la}(0,\mu_{n})\neq0$.
Let denote by
$$F_{1}(\lambda)=\frac{\sigma_{1}(0)u_{x}(0,\la)}{u(0,\la)} ~~and~~
F_{2}(\lambda)=\frac{\sigma_{2}(0)v_{x}(0,\la)}{v(0,\la)}.$$ It is
easily seen that the eigenvalues $\eta_{j}$ and $\eta'_{k}$
$(j\geq1, ~~k\geq1)$ are the poles of $F_{1}(\lambda)$ and
$F_{2}(\lambda)$, respectively. In \cite[Proposition 4]{JA}, by use of
Mittag-Lefleur theorem \cite[Chapter 4]{7}, $F_{1}(\la)$ has the following decomposition
$$F_{1}(\la)=\displaystyle\sum_{j\geq1}(\dfrac{\la}{\eta_{j}})\dfrac{c_{j}}{\la-\eta_{j}},$$
where $c_{j}$ are the residuals
of $F_{1}(\la)$ at the poles $\eta_{j}$, $j\geq1$. It is known that
the residuals of $F_{1}(\lambda)$ and $F_{2}(\lambda)$ (at the poles
$\eta_{j}$ and $\eta'_{k}$, respectively)
 are given by
$$c_{j}=\frac{\sigma_{1}(0)u_{x}(0,\eta_{j})}{\frac{\partial u }{\partial\la}(0,\eta_{j})},~~c'_{k}= \frac{\sigma_{2}
(0)v_{x}(0,\eta'_{k})}{\frac{\partial
v}{\partial\la}(0,\eta'_{k})}.$$ According to the proof of Proposition $4$ in
 \cite{JA}, we have $c_{j}>0$, $j\geq1$. By a change of
variable $s=-x$, if $v(x,\la)$ is a solution of Problem \eqref{g7},
then $v(-s, \la) $ is a solution of Problem \eqref{g6}, and hence,
$F_{2}(\la)=-\frac{\sigma_{2}(0)v_{s}(0,\la)}{v(0,\la)}$. Therefore
$c'_{k}<0$ for all $k\geq1$ and
$$\dfrac{\sigma_{1}(0)u_{x}(0,\mu_{n})}{ \frac{\partial
u}{\partial\la}(0,\mu_{n})}- \dfrac{\sigma_{2}(0) v_{x}(0,\mu_{n})}{
\frac{\partial v}{\partial\la}(0,\mu_{n})}>0,~n\geq1.$$ Passing to
the limit as $\epsilon\rightarrow 0$ in \eqref{g13}, then we prove
the first limit in \eqref{g18}. Analogously, we prove the second limit of \eqref{g18}.\\
Now, if $u(0,\mu_{n})=0$ and $v(0,\mu_{n})\neq0$ (or
$v(0,\mu_{n})=0$ and $u(0,\mu_{n})\neq0$), then
$\sigma_{1}(0)u_{x}(0,\mu_{n})\neq0$ (or
$\sigma_{2}(0)v_{x}(0,\mu_{n})\neq0$), and hence, we arrive to the
same conclusion.
\end{Proof}\\
From this lemma, it is clear that the poles of $F(\lambda)$ coincide
with the eigenvalues $(\mu_{n})_{n\geq 1}$. As a consequence of
Lemma \ref{Lg1}, it follows the following interpolation formulas
between the eigenvalues $\la_{n}$, $\mu_{n}$ and those of the
regular problem \eqref{M} for $M=0$.
\begin{Cor}\label{cg1}
 Let $\la'_{n}$, $n\geq1$ denote the eigenvalues of the regular problem \eqref{M} for $M=0$. If
 $\mu_{n}\neq\mu_{n+1}$, then
 \be
 \la_{1}<\la'_{1}<\mu_{1}~and~\mu_{n}<\la_{n+1}<\la'_{n+1}<\mu_{n+1},~n\geq1
 \label{g35}
\ee
and if $\mu_{n}=\mu_{n+1}$, then $\mu_{n}=\lambda_{n+1}=\lambda'_{n+1}$.
\end{Cor}
\begin{Proof} According to Lemma \ref{Lg1}, $F(\la)$ is a decreasing function from $+\infty$ to $-\infty$ along each of the intervals
 $(-\infty, \mu_{1})$ and $(\mu_{n}, \mu_{n+1})$, $n\geq1$. Therefore Equation \eqref{g14} has exactly
 one zero in each of these intervals. Moreover, the equation $F(\la)=0$  has exactly one zero in each of intervals
 $(-\infty, \mu_{1})$ and $(\mu_{n}, \mu_{n+1})$, $n\geq1$. It is clear that these zeros (denoted by $\la'_{n}$)
 are the eigenvalues of the regular problem \eqref{M} for $M=0$. Consequently, the interpolation properties
 are simple deductions from the curves of the functions $F(\la)$ and $M\la$ (see. Figure 1).\\
 \begin{figure}[H]
 \centering
\begin{overpic}[width=0.4\textwidth]{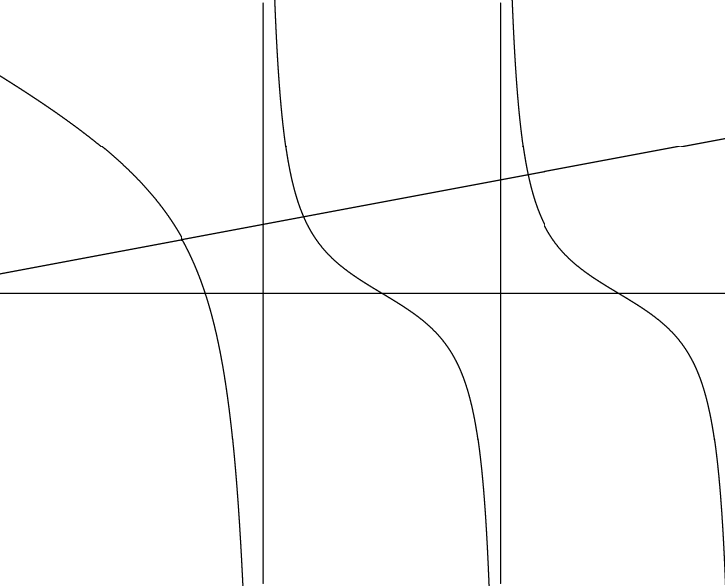}
\put(47,69.5){\bf{.}} \put(61,69.5){\bf{.}} \put(89,69.5){\bf{.}}
\put(117,69.5){\bf{.}} \put(146,69.5){\bf{.}}
\put(49,63){$\scriptscriptstyle{\la'_{1}}$}
\put(61,64){$\scriptscriptstyle{{\mu_{1}}}$}
\put(90,63){$\scriptscriptstyle{\la'_{2}}$}
\put(114,64){$\scriptscriptstyle{{\mu_{2}}}$}
\put(146,63){$\scriptscriptstyle{\la'_{3}}$} \put(41.5,82){\bf{.}}
\put(71,88){\bf{.}} \put(124,98){\bf{.}} \put(41.5,69.5){\bf{.}}
\put(41,72){\rotatebox{90}{-~-}}
\put(40,64){$\scriptscriptstyle{\la_{1}}$}
\put(70.5,65){\rotatebox{90}{~-~-~-~}}
\put(70.3,64){$\scriptscriptstyle{\la_{2}}$}
\put(123.7,65){\rotatebox{90}{~-~-~-~-}}
\put(123.8,64){$\scriptscriptstyle{\la_{3}}$}
\end{overpic}
\caption{Interpolation properties}
\end{figure}
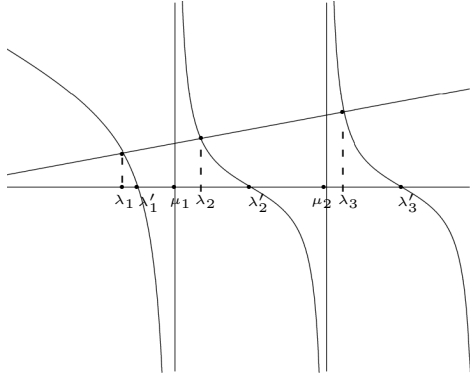
\end{Proof}\\
We are ready now to prove Theorem \ref{T1}.\\
\begin{proof} {\bf  of Theorem\ref{T1}.}
We assume that $\mu_{n}\neq \mu_{n+1}$, for $n\geq1$.
First, we establish the asymptotic estimate \be
|\la_{n+1}-\mu_{n}|=\mathcal{O}(1). \label{g19} \ee
For $\la \in (\mu_{n}, \la_{n+1}]$, we put
$G(\la)=\dfrac{1}{F(\la)}$. In view of Corollary
(\ref{cg1}) $(\mu_{n}, \la_{n+1}]\subset(\mu_{n},\la'_{n+1})$. Hence,
$\sigma_{1}(0)v(0,\la)u_{x}(0^-,\la)-\sigma_{2}(0)u(0,\la)v_{x}(0^+,\la)\neq0$
for $\la \in (\mu_{n}, \la_{n+1}]$ and this implies that $G(\la)$ is
well-defined in this interval. By use of the mean value theorem on
the interval
 $[\la_{n}^{\epsilon}, \la_{n+1}]$, (where $\la_{n}^{\epsilon}=\mu_{n}+\epsilon$ for enough small
 $\epsilon>0$),
we have
$$G(\la_{n+1})-G(\la_{n}^{\epsilon})=(\la_{n+1}-\la_{n}^{\epsilon})\frac{\partial G(\la)}{\partial \la}\mid_{\la=\alpha_{n}},$$
for some $ \alpha_{n}\in (\la_{n}^{\epsilon}, \la_{n+1})$. Using
Equation \eqref{g14} and the expression \eqref{g17} of
$\frac{\partial F(\la)}{\partial \la}$, we obtain
\bean\dfrac{1}{M\la_{n+1}}-G(\la_{n}^{\epsilon})=(\la_{n+1}-\la_{n}^{\epsilon})\frac{\partial
G}{\partial\la}(\alpha_{n}),\label{g20}\eean where
$$\frac{\partial G}{\partial\la}(\alpha_{n})=\dfrac{v^2(0,\alpha_{n})\int_{-1}^{0} \rho_{1}(s)u^{2}(s,\alpha_{n})ds+
u^2(0,\alpha_{n})\int_{0}^{1}
\rho_{2}(s)v^{2}(s,\alpha_{n})ds}{\Big(\sigma_{1}(0)v(0,\alpha_{n})u_{x}(0,\alpha_{n})-
\sigma_{2}(0)u(0,\alpha_{n})v_{x}(0,\alpha_{n})\Big)^{2}}.$$ Taking
into account the asymptotes \eqref{g21} and \eqref{g22}, a simple
calculation gives \be\label{g31} \la_{n+1}-\la_{n}^{\epsilon}=C
\alpha_{n}\left(\frac{1}{M\la_{n+1}}-G(\la_{n}^{\epsilon})\right),
\ee where $C$ is a positive constant. It is known (e.g., \cite{FVAA} and \cite[Chapter 6.7]{IGK}), that the eigenvalues $\la'_{n}$ (of Problem \eqref{M} for
$M=0$) satisfy the asymptotes
\be\label{g23}\la'_{n}=\left(\frac{n\pi}{\int_{-1}^{1}
\sqrt{\frac{\rho(x)}{\sigma(x)}}dx}\right)^{2}+\mathcal{O}(1),\ee where \bea
\rho(x)=\left\{
 \begin{array}{ll}
   \rho_{1}(x),~~x\in[-1,0], \\
   \rho_{2}(x),~~x\in[0,1],
\end{array}
 \right.
~~and~~ \sigma(x)=\left\{
 \begin{array}{ll}
   \sigma_{1}(x),~~x\in[-1,0], \\
   \sigma_{2}(x),~~x\in[0,1],
\end{array}
 \right.
. \eea Since $\la'_{n}<\alpha_{n}<\la_{n+1}<\la'_{n+1}$, then by
\eqref{g23},
\bean\alpha_{n}\sim\left(\dfrac{n\pi}{\gamma_{1}+\gamma_{2}}\right)^{2}\label{g24}
\eean and \bean
\la_{n}\sim\left(\dfrac{n\pi}{\gamma_{1}+\gamma_{2}}\right)^{2},~~as~~n\rightarrow\infty,\label{g30}\eean
where $\gamma_{1}$ and $\gamma_{2}$ are defined by \eqref{g36}. From
\eqref{g18}, we have $\la_{n}^{\epsilon}\rightarrow\mu_{n}$ and
$G(\la_{n}^{\epsilon})\rightarrow 0 $ as $\epsilon\rightarrow0$.
Then by \eqref{g31}, \eqref{g24} and \eqref{g30}, the desired estimate
\eqref{g19} follows. Since
$(\la_{n+1}-\la_{n})=(\la_{n+1}-\mu_{n})+(\mu_{n}-\la_{n})$, then by
\eqref{g19} the estimate \eqref{g26} is proved.\\ Now if $\mu_n=\mu_{n+1}$,
then by Corollary \ref{cg1}, $\mu_n=\la_{n+1}$. It is known \cite[Chapter 1]{BI}, that the
eigenvalues $\eta_{n}$ and $\eta'_{n}$ (of the boundary problems \eqref{g1} and \eqref{g2},
respectively) satisfy the asymptotics
\bean \left\{
 \begin{array}{ll}
   \eta_{n}=\(\frac{n\pi}{\gamma_{1}}\)^2+\mathcal{O}(1), \\
   \eta'_{n}=\(\frac{n\pi}{\gamma_{2}}\)^2+\mathcal{O}(1).
\end{array}
 \right.\label{g37}
\eean
Since $\mu_n$ is an eigenvalue of the both
problems \eqref{g1} and \eqref{g2}, then from \eqref{g37}, we have $\mu_n-\mu_{n-1}=\mathcal{O}(n)$.
In view of Corollary \ref{cg1}, $\la_n\in (\mu_{n-1},\mu_n)$, then from the asymptote \eqref{g19}
 we get
$$(\la_{n+1}-\la_{n})=\mathcal{O}(n).$$
The theorem is proved.
\end{proof}\\
We establish now the asymptotic behavior of the eigenfunctions
$(\phi_{n}(x))_{n\geq1}$ of the eigenvalue problem \eqref{M}. Let
as denote by $\phi_{n}^{u}(x)$ and $\phi_{n}^{v}(x)$ the restrictions
of the eigenfunctions $\phi_{n}(x)$ to $[-1,0]$ and $[0,1]$,
respectively.
\begin{Prop}\label{prop1}
Define the set
$$\Lambda=\left\{n\in\mathbb{N^{*}}~such~that~\mu_{n}\in\Gamma^{*}\right\},$$
where the set $\Gamma^{*}$ is defined by \eqref{g34}. Then the
associated eigenfunctions $(\phi_{n}(x))_{n\geq1}$
of the eigenvalue problem \eqref{M} satisfy the following asymptotic estimates:\\
\begin{description}
  \item[i$)$] For $n\in \Lambda$,
\bean\label{pro1} \left\{
\begin{array}{ll}
\phi_{n}^{u}(x)=-\left(\dfrac{\rho_{1}(x)\sigma_{1}(x)}{\rho_{2}(0)\sigma_{2}(0)}\right)^{-\frac{1}{4}}\dfrac{\cos(\sqrt{\lambda_{n}}\gamma_{2})\sin\left(\sqrt{\lambda_{n}}\int_{-1}^{x}
\sqrt{\frac{\rho_{1}(t)}{\sigma_{1}(t)}}dt\right)}{\sqrt{\lambda_{n}}\gamma_{1}\gamma_{2}}+\mathcal{O}(\frac{1}{n^{2}}),\\
\phi_{n}^{v}(x)=\left(\dfrac{\rho_{2}(x)\sigma_{2}(x)}{\rho_{1}(0)\sigma_{1}(0)}\right)^{-\frac{1}{4}}\dfrac{\cos(\sqrt{\lambda_{n}}\gamma_{1})\sin\left(\sqrt{\lambda_{n}}\int_{x}^{1}
\sqrt{\frac{\rho_{2}(t)}{\sigma_{2}(t)}}dt\right)}{\sqrt{\lambda_{n}}\gamma_{1}\gamma_{2}}+\mathcal{O}(\frac{1}{n^{2}}),
\end{array}
 \right.
\eean
where \bea
\gamma_{1}=\int_{-1}^{0}\sqrt{\frac{\rho_{1}(x)}{\sigma_{1}(x)}}dx ~~and~~~
\gamma_{2}=\int_{0}^{1}\sqrt{\frac{\rho_{2}(x)}{\sigma_{2}(x)}}dx.
\eea
  \item[ii$)$] For $n \in \mathbb{N}^{*}\backslash\Lambda$,
\bean\label{pro2} \left\{
\begin{array}{ll}
\phi_{n}^{u}(x)=\dfrac{\left(\rho_{1}(x)\sigma_{1}(x)\right)^{-\frac{1}{4}}}{\left(\rho_{2}(0)\sigma_{2}(0)\right)^{\frac{1}{4}}}
\dfrac{\sin(\sqrt{\lambda_{n}}\gamma_{2})\sin\left(\sqrt{\lambda_{n}}\int_{-1}^{x}
\sqrt{\frac{\rho_{1}(t)}{\sigma_{1}(t)}}dt\right)}{\sqrt{\lambda_{n}}\gamma_{1}\gamma_{2}}+\mathcal{O}(\frac{1}{n^{2}}),\\
\phi_{n}^{v}(x)=\dfrac{\left(\rho_{2}(x)\sigma_{2}(x)\right)^{-\frac{1}{4}}}{\left(\rho_{1}(0)\sigma_{1}(0)\right)^{\frac{1}{4}}}\dfrac{\sin(\sqrt{\lambda_{n}}\gamma_{1})
\sin\left(\sqrt{\lambda_{n}} \int_{x}^{1}
\sqrt{\frac{\rho_{2}(t)}{\sigma_{2}(t)}}dt\right)}{\sqrt{\lambda_{n}}\gamma_{1}\gamma_{2}}+\mathcal{O}(\frac{1}{n^{2}}).
\end{array}
 \right.
\eean
\end{description}
\end{Prop}
\begin{Proof}
It is clear that for all $n\in\Lambda$,
$u(0,\lambda_{n})=v(0,\lambda_{n})=0$. Then, from the last condition
in \eqref{M}, the corresponding eigenfunctions
$(\phi_{n}(x))_{n\in\Lambda}$ have the form
\begin{equation}\label{ss}
\phi_{n}(x)= \left\{
\begin{array}{lll}
\sigma_{2}(0)v_{x}(0,\lambda_{n})u(x,\lambda_{n}),&-1\leq x\leq 0,\\
\sigma_{1}(0)u_{x}(0,\lambda_{n})v(x,\lambda_{n}),&0\leq x\leq 1.\\
\end{array}
\right.
\end{equation}
Therefore the asymptotics \eqref{pro1} are simple deductions
from the asymptotics \eqref{g21}, \eqref{g22} and \eqref{ss}.
Now, if $n\in\mathbb{N^{*}}\backslash\Lambda$, then
$\phi_{n}^{u}(0)=\phi_{n}^{v}(0)\neq 0$. We set
\begin{equation}\label{gg}
\phi_{n}(x)=\sqrt{\lambda_{n}}\widetilde{U}(x,\la_{n}),
\end{equation}
where $\widetilde{U}(x,\la_{n})$ is defined by \eqref{g4}. From the
asymptotics \eqref{g21}, \eqref{g22} and \eqref{gg} we obtain the asymptotics
\eqref{pro2}.
\end{Proof}
\section{Dirichlet Boundary Control}\label{dc}
In this section, we study the boundary null-controllability of
Problem (\ref{a})-(\ref{b}). We do it by reducing the control
problem to problem of moments. Then, we will solve this problem
of moments using the theory developed in \cite{DH, DH1}.
Namely, we enunciate the following result:
\begin{Theo}\label{DC1}
Assume that the coefficients $\sigma_{i}(x)$, $\rho_{i}(x)$ and
$q_{i}(x)$ $(i=1,2)$ satisfy \eqref{i1} and \eqref{i2}. Let $ T >
0$, then for any $Y^{0}=(u^{0},v^{0},z^{0})^t\in \mathcal{H}$ there
exists a control $h(t)\in L^{2}(0, T )$ such that the solution
$Y(t,x)=(u(t,x),v(t,x),z(t))^t$ of the control system
(\ref{a})-(\ref{b}) satisfies
 $$u(T,x)=v(T,x)=z(T)=0.$$
\end{Theo}
For the proof of this theorem, we establish some preliminary results.
Let us introduce the adjoint problem to the control system (\ref{a})-(\ref{b})
\bean\label{cda} \left\{
 \begin{array}{ll}
   -\rho_{1}(x)\widetilde{u}_{t}=(\sigma_{1}(x)\widetilde{u}_{x})_{x}-q_{1}(x)\widetilde{u},&~~ x\in(-1,0), ~~t>0, \\
   -\rho_{2}(x)\widetilde{v}_{t}=(\sigma_{2}(x)\widetilde{v}_{x})_{x}-q_{2}(x)\widetilde{v},&~~ x\in(0,1), ~~~~t>0, \\
   -M \widetilde{z}_{t}(t)=\sigma_{2}(0)\widetilde{v}_{x}(0,t)-\sigma_{1}(0)\widetilde{u}_{x}(0,t),&~~ t>0, \\
   \widetilde{u}(0,t)=\widetilde{v}(0,t)=\widetilde{z}(t),&~~ t>0,\\
   \widetilde{u}(-1,t)=0, \widetilde{v}(1,t)=0,&~~ t>0, \\
\end{array}
 \right.
\eean with terminal data at $t = T>0$ given by \be \label{cdv} \left\{
\begin{array}{l}
  \widetilde{u}(x,T)=\widetilde{u}^{T}(x),~~x\in(-1,0), \\
  \widetilde{v}(x,T)=\widetilde{v}^{T}(x),~~x\in(0,1), \\
  \widetilde{z}(T)=\widetilde{z}^{T}.\\
\end{array}
 \right.
\ee By letting $\widetilde{Y} = (\widetilde{u}, \widetilde{v},
\widetilde{z})^t$, the above problem \eqref{cda}-\eqref{cdv} can be written as a Cauchy
problem as \be
\dot{\widetilde{Y}}(t)=\mathcal{A}\widetilde{Y}(t),~~\widetilde{Y}^{T}=\widetilde{Y}(T),~~t>0.
\label{cc2}  \ee We start by mentioning the well-posedness of the
Cauchy problem (\ref{a})-(\ref{b}). Let us recall from Proposition \ref{rr}, the operator $\mathcal{A}$ is
positive and selfadjoint, and hence, it generates a scale of Hilbert
spaces $\mathcal{H}_{\theta}=D(\mathcal{A}^{\theta})$, $\theta\geq0
$. In particular
 $\mathcal{H}_{0}=\mathcal{H}$ and $\mathcal{H}_{\frac{1}{2}}=D(\mathcal{A}^{\frac{1}{2}})$.
The norm in $\mathcal{H}_{\frac{1}{2}}$ is given by \bea
\|y\|_{\frac{1}{2}}^{2}&=&\langle\mathcal{A}^{\frac{1}{2}}y,
\mathcal{A}^{\frac{1}{2}}y\rangle_\mathcal{H}\\
&=&\int_{-1}^{0}\sigma_{1}(x)|u^{'}|^{2}+q_{1}(x)|u|^{2}dx
+\int_{0}^{1}\sigma_{2}(x)|v^{'}|^{2}+q_{2}(x)|v|^{2}dx\eea Our
assumptions on $\sigma_{i}(x)$ and $q_{i}(x)$ $(i=1,2)$ imply that,
$\mathcal{H}_{\frac{1}{2}}$ is topologically equivalent to the subspace
$\mathcal{W}$ (where $\mathcal{W}$ is defined by \eqref{v}). Let $\widetilde{Y}^{T}\in\mathcal{H}_{\frac{1}{2}}$ then $\widetilde{Y}\in
C([0,T], \mathcal{H}_{\frac{1}{2}})$ and is given by
\be{\widetilde{Y}}(t)=\mathbf{S}(T-t)\widetilde{Y}^{T},~~0\leq t
\leq T. \label{cd3}\ee Let $Y$ be a weak solution of the control
problem (\ref{a})-(\ref{b}) with $h(t)\in L^{2}(0, T)$, and
$\widetilde{Y}$ be a solution of the adjoint problem \eqref{cd3}.
Then \bea&\displaystyle
\int_{-1}^{0}\int_{0}^{T}\left(\rho_{1}(x)u_{t}(x,t)-(\sigma_{1}(x)u_{x}(x,t))_{x}+
q_{1}(x)u(x,t)\right)\widetilde{u}(x,t)dxdt+\\
&~~~~~~\displaystyle\int_{0}^{1}\int_{0}^{T}\left(\rho_{2}(x)v_{t}(x,t)-
(\sigma_{2}(x)v_{x}(x,t))_{x}+q_{2}(x)v(x,t)\right)\widetilde{v}(x,t)dxdt=0\eea
By a performing integrations by parts, we obtain the identity \bea
<Y(T),\widetilde{Y}(T)>_{\mathcal{H}_{-\frac{1}{2}},~\mathcal{H}_{\frac{1}{2}}}=
<Y(0),\widetilde{Y}(0)>_{\mathcal{H}}
-\sigma_{2}(1)\int_{0}^{T}h(t)\widetilde{v}_{x}(1,t)dt. \eea
Equivalently, \bean \label{tr}
<Y(T),\widetilde{Y}^{T}>_{\mathcal{H}_{-\frac{1}{2}},~\mathcal{H}_{\frac{1}{2}}}=
<Y(0),\mathbf{S}_{T}\widetilde{Y}^{T}>_{\mathcal{H}}
-\sigma_{2}(1)\int_{0}^{T}h(t)\widetilde{v}_{x}(1,t)dt, \eean where
$<.,.>$ denotes the duality product between the two spaces
$\mathcal{H}_{-\frac{1}{2}}$ and $\mathcal{H}_{\frac{1}{2}}$. Let us
recall that, with this regularity on $h$ and $Y^{0}\in
\mathcal{H}_{-\frac{1}{2}}$, the Cauchy problem (\ref{a})-(\ref{b})
is well posed in the transposition sense (See, e.g., \cite{JE1}) in $C^{0}([0,T],~
\mathcal{H}_{-\frac{1}{2}})$: it has only one solution in
this set and there exists a constant $C
> 0$ independent of $h(t)\in L^{2}(0, T)$ and $Y^{0}\in
\mathcal{H}_{-\frac{1}{2}}$ such that
$$\|Y\|_{L^{\infty}(0,T;~\mathcal{H}_{-\frac{1}{2}})}\leq
C(\|Y^{0}\|_{\mathcal{H}_{-\frac{1}{2}}}+\|h\|_{L^{2}(0, T)}).$$
This result is basically well known (See, e.g., \cite[Section 2.7.1]{J}).
For $\sigma_{i}(x)$ and $q_{i}(x)$ $(i=1,2)$ constant on each
interval, the result is proven in \cite{HM} and also the proof can
be extended to the variable coefficient case. The following Lemma
characterizes the problem of null controllability of
(\ref{a})-(\ref{b}) in terms of the solution $\widetilde{Y}$ of the
adjoint problem \eqref{cc2}.
\begin{Lem}\label{L1}The control problem (\ref{a})-(\ref{b}) is null controllable in
time $T > 0$ if and only if, for any $Y^{0}\in \mathcal{H}$ there
exists $h(t) \in L^{2}(0, T )$ such that the following relation
holds \be \label{cd1}
<Y^{0},\mathbf{S}_{T}\widetilde{Y}^{T}>_{\mathcal{H}}=\sigma_{2}(1)\int_{0}^{T}h(t)\widetilde{v}_{x}(1,t)dt,
\ee for any $\widetilde{Y}^{T}\in\mathcal{H}$, where $\widetilde{Y}$
is the solution to the adjoint problem \eqref{cc2}.
\end{Lem}
\begin{Proof}
First we assume that \eqref{cd1} is verified. Then by \eqref{tr} it
follows that $Y(T)=0$. Hence, the solution is controllable to zero and $h(t)$ is a control of Problem (\ref{a})-(\ref{b}).\\
Reciprocally, if $h(t)$ is a control of Problem (\ref{a})-(\ref{b}), we
have that $Y(T) = 0$. From \eqref{tr} it follows that \eqref{cd1}
holds.
\end{Proof}\\
From the previous Lemma we deduce the following result:
\begin{Prop}\label{xx} Problem (\ref{a})-(\ref{b}) is null-controllable in time $T > 0$ if and
only if for any $Y^{0} \in \mathcal{H}$, with Fourier expansion
\bea Y^{0}&=&\sum_{n\in\mathbb{N}^{*}}Y_{n}^{0}\widetilde{\phi}_{n}(x),\eea there
exists a function $w(t)\in L^{2}(0,T)$ such that for all
$n\in\mathbb{N}^{*}$ \be\frac{Y_{n}^{0}
\|\widetilde{\phi}_{n}\|^{2}}{\sigma_{2}(1)(\phi_{n}^{v})_{x}(1)}e^{-\la_{n}T}=
\int_{0}^{T}w(t)e^{-\la_{n}t}dt.\label{m1}\ee The control is given
by $w(t):=h(T-t)$.
\end{Prop}
\begin{Proof} From the previous Lemma, it follows that $h(t)\in L^{2}(0,T)$ is a control of Problem (\ref{a})-(\ref{b})
if and only if it satisfies \eqref{cd1}. But, since the
eigenfunctions $(\widetilde{\phi}_{n}(x))_{n\in\mathbb{N}^{*}}$  of the operator
$\mathcal{A}$ forms an orthonormal basis in $\mathcal{H}$, we have
that any initial data $Y^{0}\in \mathcal{H}$ for the control problem
(\ref{a})-(\ref{b}) can be expressed as
\bea Y^{0}=\sum_{n\in\mathbb{N}^{*}}Y_{n}^{0}\widetilde{\phi}_{n}(x)\eea where
the Fourier coefficients $Y_{n}^{0}= \langle Y^{0}, \widetilde{\phi}_{n}
\rangle_{\mathcal{H}}$, $n\in\mathbb{N}^{*}$, belong to $\ell^2$. On
the other hand, we observe that
$$\widetilde{Y}_{n}(t,x)=(\widetilde{u}_n(t,x), \widetilde{v}_n(t,x),
\widetilde{z}_n(t,x))^t=e^{-\la_{n}(T-t)}\widetilde{\phi}_n(x)$$ is the
eigensolution of the adjoint problem \eqref{cc2}. 
We put these solutions in Equation \eqref{cd1} to obtain the moment problem \eqref{m1}. The proof ends by taking
$w(t)=h(T-t)$.
\end{Proof}\\
We are now in a position to prove Theorem \ref{DC1}.\\
\begin{proof} {\bf of Theorem \ref{DC1}.}
From the asymptotics \eqref{pro1} and \eqref{pro2}, it is easy to
show that for large $n\in\mathbb{N}^{*}$ there exists a constant
$\mathbf{c}>0$ such that  \bean
\|\widetilde{\phi}_{n}(x)\|^{2}=\|\phi_{n}^{u}(x)\|_{L^{2}[-1,0]}^{2}+
\|\phi_{n}^{v}(x)\|_{L^{2}[0,1]}^{2}+
M|\phi_{n}^{v}(0)|^{2} \leq \frac{\mathbf{c}}{n^{2}}.\label{ccc1}\eean
It is clear from the initial conditions in \eqref{g7} and the
expression \eqref{ss}, if $n\in \Lambda$ then
$(\phi^{v}_{n})_{x}(1)=-\sigma_{1}(0)u_{x}(0,\lambda_{n})$. Since
$u(0,\lambda_{n})=0$ for $n\in \Lambda$, then
$u_{x}(0,\lambda_{n})\neq0$. Using \eqref{g21}, we obtain
$(\phi^{v}_{n})_{x}(1)=-(\sigma_{1}(0)\rho_{1}(0))^{\frac{1}{4}}
\cos(\sqrt{\lambda_{n}}\gamma_{1})+\mathcal{O}(\frac{1}{n})$.
Since the eigenvalues $\la_n$ for $n\in \Lambda$, satisfy the both asymptotics in \eqref{g37}, then
\be(\phi^{v}_{n})_{x}(1)=\mathcal{O}(1),~~ for ~~n\in \Lambda.\label{ccc2}\ee\\
Now let $n \in \mathbb{N}^{*}\backslash\Lambda$. Then $u(0,\la_n)\neq0$ since otherwise $\la_n$ would be an eigenvalue of Problem \eqref{g1}. From the expression \eqref{gg}, we have
 $(\phi_{n}^{v})_{x}(1)=\sqrt{\la_{n}}u(0,\la_{n})v_x(1,\la_{n})$. In view of
 the initial conditions in \eqref{g7}, $(\phi_{n}^{v})_{x}(1)=\sqrt{\la_{n}}u(0,\la_{n})$.
 By \eqref{g21}, we obtain
\be
(\phi_{n}^{v})_{x}(1)=\frac{-(\rho_{1}(0)\sigma_{1}(0))^{-\frac{1}{4}}}{\gamma_{1}}
\sin\left(\sqrt{\la_{n}}\gamma_{1}\right)+\mathcal{O}(\frac{1}{n}).\label{ccc3}
\ee
By use of \eqref{ccc1}, \eqref{ccc2}, \eqref{ccc3} and the asymptotic \eqref{g30}, we deduce
that for large $n\in\mathbb{N}^{*}$ there exists a constants $C_1>0$ and $C_2>0$ such that
$$\Big{|}\frac{Y_{n}^{0} \|\phi_{n}\|^2}{\sigma_{2}(1)(\phi_{n}^{v})_{x}(1)}\Big{|}e^{-\la_{n}T}\leq
C_1 e^{-C_2n^2}.$$ On the other hand, by the asymptotic \eqref{g30}, we have
\bean\sum_{n\in \mathbb{N}^{*}}\frac{1}{\la_{n}}<\infty.\label{cd4}\eean
From \eqref{cd4} together with Theorem \ref{T1}, there exists a
biorthogonal sequence $(\theta_{n}(t))_{n\in \mathbb{N}^{*}}$ to the
family of exponential functions $(e^{-\la_{n}t})_{n\in
\mathbb{N}^{*}}$ (see \cite{DH1,L}) such that
$$\int_{0}^{T}\theta_{n}(t)e^{{-\la_{m}t}}dt=\delta_{nm}=\left\{
\begin{array}{l}
                                                      1,~~if~~n=m, \\
                                                      0,~~if~~n \neq m.
                                                    \end{array}
 \right.
$$
Again by \eqref{g30} together with the
general theory developed in \cite{DH}, there exists constants $M>0$ and $\omega>0$
such that
$$\parallel\theta_{m}(t)\parallel_{L^{2}(0,T)}\leq M
e^{\omega m},~m \in \mathbb{N}^{*}.$$ It is easy to see that the
above implies the convergence of the series $$w(t)= \sum_{n\in
\mathbb{N}^{*}}\frac{Y_{n}^{0}
\|\phi_{n}\|^{2}}{\sigma_{2}(1)(\phi_{n}^{v})_{x}(1)}e^{-\la_{n}T}\theta_{n}(t)$$
which provides a solution to the moment problem \eqref{m1}.
Therefore, Theorem \ref{DC1} is a direct consequence of
the convergence of the series $w(t)$ in $L^2(0,T)$ and Proposition \ref{xx}. 
\end{proof}
\section{Neumann Boundary Control}
In this section we study the Neumann boundary null-controllability
of Problem \eqref{a}-\eqref{c}. We consider Problem \eqref{a}
with the homogeneous Neumann boundary condition
\be\label{w2} v_{x}(1,t)=0.\ee
As before, we define $
\mathcal{H}=L^{2}(-1,0)\times L^{2}(0,1)\times \mathbb{R}$ with the
scalar product defined by \eqref{dv}. We set \bean
\mathcal{V}=\{(u,v)\in \mathcal{V}_{1}\times
\mathcal{V}_{2}~~|~~u(0)=v(0)\},\nonumber\eean where
$\mathcal{V}_{2}=H^{1}(0,1)$ and $\mathcal{V}_{1}$ is defined by
\eqref{w3}. Define
\begin{equation}\label{vn}
\mathcal{W}=\{(u,v,z)\in \mathcal{V} \times \mathbb{R}:
u(0)=v(0)=z\},
\end{equation}
equipped with norm
$\|(u,v,z)\|^{2}_{\mathcal{W}}=\|(u,v)\|^{2}_{\mathcal{V}}$. We
associate System \eqref{a}-\eqref{w2} with a self-adjoint operator
$\mathcal{A}$ defined in $\mathcal{H}$ by \eqref{est}. The domain
$D(\mathcal{A})$ of $\mathcal{A}$ is dense in $\mathcal{H}$ and is
given by
$$
D(\mathcal{A})=\{Y=(u, v, z)\in \mathcal{W}~:~ (u,v)\in
H^{2}(-1,0)\times H^{2}(0,1),~~v_{x}(1)=0\}.
$$
Now, we consider the following spectral problem associated with
System (\ref{a})-(\ref{w2}) \bean\label{N} \left\{
 \begin{array}{ll}
-(\sigma_{1}(x)u')'+q_{1}(x)u=\nu \rho_{1}(x)u,~~&x\in (-1,0), \\
-(\sigma_{2}(x)v')'+q_{2}(x)v=\nu \rho_{2}(x)v,~~&x\in(0,1), \\
u(-1)=v'(1)=0,~u(0)=v(0), \\
\sigma_{1}(0)u'(0)-\sigma_{2}(0)v'(0)=\nu M u(0). \\
\end{array}
 \right.
\eean
Following the idea of the proof of Lemma \ref{sss}, we can prove that the
spectrum of System \eqref{N} is discrete and it consists of an
increasing sequence of positive and simple
eigenvalues $(\nu_{n})_{n\in\mathbb{N}^{*}}$  tending to $+\infty$\\
$$0<\nu_{1}<\nu_{2}<.......<\nu_{n}<.....\underset{n\rightarrow \infty}{\longrightarrow}\infty .$$
Moreover, the corresponding eigenfunctions
$(\widetilde{\phi}_{n}(x))_{n\in\mathbb{N^{*}}}$ form an orthonormal basis in
$\mathcal{H}$. System \eqref{a} with the homogeneous
Neumann boundary condition \eqref{w2} can be written as
\begin{equation}
\dot{Y}(t)=-\mathcal{A}Y(t),~~Y^{0}=Y(0),~~ t>0,
\end{equation}
where $Y^{0}=(u(0),v(0),z(0))^{t}$. The well-posedness of System \eqref{a}-\eqref{w2} can be analyzed similarly to that
of the Dirichlet case described in Section \ref{s1}.\\ We discuss now the
asymptotic behavior of the spectral gap $\nu_{n+1}-\nu_{n}$ for
large $n$.
\begin{Theo}\label{Tn1}
For all $n\in \mathbb{N^{*}}$, there is a constant $\delta_{2}>0$
such that the sequence of eigenvalues $(\nu_{n})_{n\geq1}$ of the
spectral problem \eqref{N} satisfy the asymptotic \bean
\nu_{n+1}-\nu_{n}\geq\delta_{2},~n\in \mathbb{N}^{*}\label{gn26}.
\eean
\end{Theo}
Let
$\Gamma=\{\mu_{n}\}_{1}^{\infty}=\{\eta_{j}\}_{1}^{\infty}\bigcup\{\zeta'_{k}\}_{1}^{\infty}$
where $\eta_{j}$ and $\zeta'_{k}$ are the eigenvalues of Problem
\eqref{g1} and the following boundary value problem
\begin{equation}\label{gn2}
\left\{
\begin{array}{lll}
-(\sigma_{2}(x)y')'+q_{2}(x)y=\nu\rho_{2}(x)y, x\in (0,\,1),\\
y(0)=y'(1)=0,
\end{array}
\right.
\end{equation}
respectively. Obviously, $\eta_{j}$ and $\zeta'_{k}$ can be coincide. Let
\be\Gamma^{*}=\{\mu_{n}\in\Gamma~\backslash~
\eta_{j}=\zeta'_{k},~j,~k\in\mathbb{N^{*}}\}.\label{gn34}\ee
We consider the following boundary value problem
\begin{equation}\label{gn3}
\left\{
\begin{array}{lll}
-(\sigma_{1}(x)u')'+q_{1}(x)u=\nu \rho_{1}(x) u, x\in (-1,\,0),\\
-(\sigma_{2}(x)v')'+q_{2}(x)v=\nu \rho_{2}(x)v, x\in (0,\,1),\\
u(-1)=v'(1)=0,\\
u(0)=v(0).\\
\end{array}
\right.
\end{equation}
It is clear that for $\nu \in \mathbb{C}\backslash\Gamma^{*}$, the
set of solutions of Problem \eqref{gn3} is one-dimensional
subspace which is generated by a solution of the form
\begin{equation}\label{gn4}
\widetilde{U}(x,\nu)= \left\{
\begin{array}{lll}
{v(0,\nu)}u(x,\nu),&-1\leq x\leq 0,\\
{u(0,\nu)}v(x,\nu),&0\leq x\leq 1,\\
\end{array}
\right.
\end{equation}
where $u(x,\nu)$ and $v(x,\nu)$ are the solutions of the initial value problems
\begin{equation}\label{w111}
\left\{
\begin{array}{lll}
-(\sigma_{1}(x)u')'+q_{1}(x)u=\nu \rho_{1}(x) u, ~~x\in (-1,\,0),\\
u(-1)=0,~u'(-1)=\sqrt{\nu}
\end{array}
\right.
\end{equation}and
 \begin{equation}\label{gn7}
\left\{
\begin{array}{lll}
-(\sigma_{2}(x)v')'+q_{2}(x)v=\nu \rho_{2}(x)v,~~x\in (0,\,1), \\
v'(1)=0,v(1)=1,
\end{array}
\right.
\end{equation}
respectively.
Note that $u(0,\nu)\neq0$ and $v(0,\nu)\neq0$ for $\nu
\in(\mu_{n},\mu_{n+1})$, since otherwise $\nu$ would be an
eigenvalue of one of Problems \eqref{g1} and \eqref{gn2}. As in
Section \ref{GD1}, we introduce the related meromorphic function \be
F(\nu)=\frac{\sigma_{1}(0)v(0)u_{x}(0^-)-\sigma_{2}(0)u(0)v_{x}(0^+)}{u(0)v(0)},~~\nu\in\mathbb{C}\backslash\Gamma.\label{gn5}\ee
Clearly, the eigenvalues $\nu_{n}$, $n\geq1$ of Problem \eqref{N}
are the solutions of the equation \be F(\nu)=M\nu\label{gn14}.\ee
\begin{Lem}\label{Lgn1}
The function $F(\nu)$ is decreasing along the intervals
$(-\infty,\mu_{1})$ and $(\mu_{n}, \mu_{n+1})$, $n\geq1$ (with
$\mu_{n}\neq\mu_{n+1}$). Furthermore, it decreases from $+\infty$
to $-\infty$.
\end{Lem}
\begin{Proof}
The first statement about the monotonicity decreasing of the
function $F(\nu)$ can be shown in a same way as the proof of Lemma
\ref{Lg1}. It is known (e.g., \cite[Chapter 1]{BI} and
\cite[Chapter 2]{M}), for
$\nu\in\mathbb{C}$ and $|\nu|\rightarrow \infty$ that
\bean\label{gnw21}
\begin{cases}
u(x,\nu)=
(\rho_{1}(x)\sigma_{1}(x))^{-\frac{1}{4}}\sin(\sqrt{\nu}\int_{-1}^{x}
\sqrt{\frac{\rho_{1}(t)}{\sigma_{1}(t)}})[1],\\
 u_{x}(x,\nu)=
\sqrt{\nu}(\rho_{1}(x))^{\frac{1}{4}}(\sigma_{1}(x))^{-\frac{3}{4}}\cos(\sqrt{\nu}\int_{-1}^{x}
\sqrt{\frac{\rho_{1}(t)}{\sigma_{1}(t)}}dt)[1]
\end{cases}
\eean and \bean\label{gnw22}
\begin{cases}
v(x,\nu)=
(\rho_{2}(x)\sigma_{2}(x))^{-\frac{1}{4}}\cos(\sqrt{\nu}\int_{x}^{1}
\sqrt{\frac{\rho_{2}(t)}{\sigma_{2}(t)}}dt)[1],\\
v_{x}(x,\nu)=
-\sqrt{\nu}(\rho_{2}(x))^{\frac{1}{4}}(\sigma_{2}(x))^{-\frac{3}{4}}\sin
(\sqrt{\nu}\int_{x}^{1}
\sqrt{\frac{\rho_{2}(t)}{\sigma_{2}(t)}}dt)[1],
\end{cases}
\eean where $[1]=1+\mathcal{O}(\frac{1}{\sqrt{|\nu|}})$. By use of
\eqref{gnw21} and \eqref{gnw22}, a straightforward calculation gives the
following asymptotic \be
\label{gn15}F(\nu)\sim\sqrt{|\nu|}\big(
\sqrt{\rho_{1}(0)\sigma_{1}(0)}+\sqrt{\rho_{2}(0)\sigma_{2}(0)}\big),
~~as~~\nu\rightarrow-\infty.\ee The proof of the two limits \be
\lim_{\nu\rightarrow\mu_{n}+0}F(\nu)=+\infty,~~~~
\lim_{\nu\rightarrow\mu_{n}-0}F(\nu)=-\infty,\label{gn18}\ee is
similar to that of Lemma \ref{Lg1}.
\end{Proof}\\
From this Lemma, it is clear that the poles and the zeros of $F(\nu)$
coincide with the eigenvalues $(\mu_{n})_{n\geq1}$ and the eigenvalues $(\nu'_{n})_{n\geq1}$ of the regular problem \eqref{N}
( for $M=0$ ). As a consequence of Lemma \ref{Lgn1}, it follows the
following interpolation formulas:
\begin{Cor}\label{cgn1}
Let $\nu'_{n}$, $n\geq1$ denote the eigenvalues of the regular problem \eqref{N} for $M=0$. If
 $\mu_{n}\neq\mu_{n+1}$, then
 \bea
 \nu_{1}<\nu'_{1}<\mu_{1}~and~\mu_{n}<\nu_{n+1}<\nu'_{n+1}<\mu_{n+1},~n\geq1,
\eea
and if $\mu_{n}=\mu_{n+1}$, then $\mu_{n}=\nu_{n+1}=\nu'_{n+1}$.
\end{Cor}
\begin{proof} {\bf  of Theorem \ref{Tn1}.}
The proof of this theorem is similar to that of Theorem
\ref{T1}. We have only to use instead of the asymptotes
\eqref{g23}, \eqref{g24} and \eqref{g30}, the following asymptotic
estimates (e.g., \cite{FVAA} and \cite[Chapter 6.7]{IGK})
\be\label{gn23}\nu'_{n}=\left(\frac{(n+\frac{1}{2})\pi}{\int_{-1}^{1}
\sqrt{\frac{\rho(x)}{\sigma(x)}}dx}\right)^{2}+\mathcal{O}(1),\ee where \bea
\rho(x)=\left\{
 \begin{array}{ll}
   \rho_{1}(x),~~x\in[-1,0], \\
   \rho_{2}(x),~~x\in[0,1],
\end{array}
 \right.
~and~ \sigma(x)=\left\{
 \begin{array}{ll}
   \sigma_{1}(x),~~x\in[-1,0], \\
   \sigma_{2}(x),~~x\in[0,1].
\end{array}
 \right.
 \eea \bean\alpha_{n}\sim\left(\dfrac{(2n+1)\pi}{2(\gamma_{1}+\gamma_{2})}\right)^{2}~and~
\nu_{n}\sim\left(\dfrac{(2n+1)\pi}{2(\gamma_{1}+\gamma_{2})}\right)^{2},~as~n\rightarrow\infty.\label{gn30}\eean
\end{proof}\\
We give now the asymptotic behaviors of the
eigenfunctions $(\phi_{n}(x))_{n\geq1}$ of Problem \eqref{N}.
As in Section \ref{GD1}, let $\phi_{n}^{u}(x)$ and $\phi_{n}^{v}(x)$ be the restrictions
of the eigenfunctions $\phi_{n}(x)$ to $[-1,0]$ and $[0,1]$,
respectively.
\begin{Prop}
Define the set
$$\Lambda=\left\{n\in\mathbb{N^{*}}~such~that~\mu_{n}\in\Gamma^{*}\right\},$$
where the set $\Gamma^{*}$ is defined by \eqref{gn34}. Then the
associated eigenfunctions $(\phi_{n}(x))_{n\geq1}$ of the eigenvalue
problem \eqref{N} satisfy the following asymptotic estimates:\\
\begin{description}
  \item[i$)$] For $n\in \Lambda$,
  \bean\label{pron1} \left\{
\begin{array}{ll}
\phi_{n}^{u}(x)=-\left(\dfrac{\rho_{1}(x)\sigma_{1}(x)}{\rho_{2}(0)\sigma_{2}(0)}\right)^{-\frac{1}{4}}
\sin(\sqrt{\nu_{n}}\gamma_{2})\sin\left(\sqrt{\nu_{n}}\int_{-1}^{x}
\sqrt{\frac{\rho_{1}(t)}{\sigma_{1}(t)}}dt\right)+\mathcal{O}(\frac{1}{n})\\
\phi_{n}^{v}(x)=\left(\dfrac{\rho_{2}(x)\sigma_{2}(x)}{\rho_{1}(0)\sigma_{1}(0)}\right)^{-\frac{1}{4}}
\cos(\sqrt{\nu_{n}}\gamma_{1})\cos\left(\sqrt{\nu_{n}}\int_{x}^{1}
\sqrt{\frac{\rho_{2}(t)}{\sigma_{2}(t)}}dt\right)+\mathcal{O}(\frac{1}{n}).
\end{array}
 \right.
\eean
  \item[ii$)$] For $n \in \mathbb{N}^{*}\backslash\Lambda$,
  \bean\label{pron2} \left\{
\begin{array}{ll}
\phi_{n}^{u}(x)=\dfrac{\left({\rho_{1}(x)\sigma_{1}(x)}\right)^
{-\frac{1}{4}}}{\left(\rho_{2}(0)\sigma_{2}(0)\right)^{\frac{1}{4}}}\cos\left(\sqrt{\nu_{n}}\gamma_{2}\right)
\sin\left(\sqrt{\nu_{n}}\int_{-1}^{x}{
\sqrt{\frac{\rho_{1}(t)}{\sigma_{1}(t)}}dt}\right)+\mathcal{O}(\frac{1}{n}),\\
\phi_{n}^{v}(x)=\dfrac{\left({\rho_{2}(x)\sigma_{2}(x)}\right)^
{-\frac{1}{4}}}{\left(\rho_{1}(0)\sigma_{1}(0)\right)^{\frac{1}{4}}}\sin\left(\sqrt{\nu_{n}}\gamma_{1}\right)
\cos\left(\sqrt{\nu_{n}}\int_{x}^{1}
\sqrt{\frac{\rho_{2}(t)}{\sigma_{2}(t)}}dt\right)+\mathcal{O}(\frac{1}{n}).
\end{array}
 \right.
\eean
\end{description}
\end{Prop}
\begin{Proof}
The proof of the asymptotics \eqref{pron1} and \eqref{pron2}
is similar to that of Proposition \ref{prop1}. Here we have only to use
the asymptotics \eqref{gnw21}, \eqref{gnw22} and the following expressions of
$(\phi_{n}(x))_{n \in \mathbb{N^{*}}}$:
\bean\label{wn}
(\phi_{n}(x))_{n \in \Lambda}= \left\{
\begin{array}{lll}
\frac{\sigma_{2}(0)}{\sqrt{\nu_n}}v_{x}(0,\nu_{n})u(x,\nu_{n}),&-1\leq x\leq 0,\\
\frac{\sigma_{1}(0)}{\sqrt{\nu_n}}u_{x}(0,\nu_{n})v(x,\nu_{n}),&0\leq x\leq 1,\\
\end{array}
\right.
\eean
\begin{equation}\label{ggwn}
(\phi_{n}(x))_{n \in \mathbb{N^{*}}\backslash\Lambda}=\widetilde{U}(x,\nu_{n}),
\end{equation}
where $\widetilde{U}(x,\nu_{n})$ is defined by \eqref{gn4}.
\end{Proof}\\
We may now state our main result.
\begin{Theo}\label{nc1}
Assume that the coefficients $\sigma_{i}(x)$, $\rho_{i}(x)$ and
$q_{i}(x)$ $(i=1,2)$ satisfy \eqref{i1} and \eqref{i2}. Let $ T >
0$, then for any $Y^{0}=(u^{0},v^{0},z^{0})^t\in \mathcal{H}$ there
exists a control $h(t)\in L^{2}(0, T )$ such that the solution
$Y(t,x)=(u(t,x),v(t,x),z(t))^t$ of the control system
(\ref{a})-(\ref{c}) satisfies
 $$u(T,x)=v(T,x)=z(T)=0.$$
\end{Theo}
Let us introduce the adjoint problem to the control system
(\ref{a})-(\ref{c}) \bean\label{nc2} \left\{
 \begin{array}{ll}
   -\rho_{1}(x)\widetilde{u}_{t}=(\sigma_{1}(x)\widetilde{u}_{x})_{x}-q_{1}(x)\widetilde{u},&~~ x\in(-1,0), ~~t>0, \\
   -\rho_{2}(x)\widetilde{v}_{t}=(\sigma_{2}(x)\widetilde{v}_{x})_{x}-q_{2}(x)\widetilde{v},&~~ x\in(0,1), ~~~~t>0, \\
   -M \widetilde{z}_{t}(t)=\sigma_{2}(0)\widetilde{v}_{x}(0,t)-\sigma_{1}(0)\widetilde{u}_{x}(0,t),&~~ t>0, \\
   \widetilde{u}(0,t)=\widetilde{v}(0,t)=\widetilde{z}(t),&~~ t>0,\\
   \widetilde{u}(-1,t)=0, \widetilde{v}_{x}(1,t)=0,&~~ t>0, \\
\end{array}
 \right.
\eean with terminal data at $t = T>0$ is given by \bean \label{cd5}\left\{
\begin{array}{ll}
  \widetilde{u}(x,T)=\widetilde{u}^{T}(x),~~x\in(-1,0), \\
  \widetilde{v}(x,T)=\widetilde{v}^{T}(x),~~x\in(0,1), \\
  \widetilde{z}(T)=\widetilde{z}^{T}.
\end{array}
 \right.
\eean By letting $\widetilde{Y} = (\widetilde{u},
\widetilde{v},\widetilde{z})^t$, the above problem \eqref{nc2}-\eqref{cd5} can be written as a Cauchy problem as \be
\dot{\widetilde{Y}}(t)=\mathcal{A}\widetilde{Y}(t),~~\widetilde{Y}^{T}=\widetilde{Y}(T)\in
\mathcal{H},~~t>0. \label{nc3}\ee Then $\widetilde{Y}\in
C([0,T],~\mathcal{H})$ and is given by
\be{\widetilde{Y}}(t)=\mathbf{S}(T-t)\widetilde{Y}^{T},~~0\leq t
\leq T. \label{nc4}\ee Let $Y$ be a weak solution of the control
problem (\ref{a})-(\ref{c}) with $h(t)\in L^{2}(0, T)$. Then by
integrations by parts, we obtain \bean \label{n1}
<Y(T),\widetilde{Y}^{T}>_{\mathcal{H}}=
<Y(0),\mathbf{S}_{T}\widetilde{Y}^{T}>_{\mathcal{H}}
+\sigma_{2}(1)\int_{0}^{T}h(t)\widetilde{v}(1,t)dt. \eean
The Cauchy problem (\ref{a})-(\ref{c}) is well posed in the set
$C([0,T], \mathcal{H})$: it has only one solution in
this set and there exists a constant $C > 0$ independent of $h(t)\in
L^{2}(0, T)$ and $Y^{0}\in \mathcal{H}$ such that
$$\|Y\|_{L^{\infty}(0,T;\mathcal{H})}\leq
C(\|Y^{0}\|_{\mathcal{H}}+\|h\|_{L^{2}(0, T)}).$$ As in Section \ref{dc}, we have
the following lemma.
\begin{Lem}The control problem (\ref{a})-(\ref{c}) is null controllable in
time $T > 0$ if and only if, for any $Y^{0}\in \mathcal{H}$ there
exists $h(t) \in L^{2}(0, T )$ such that the following relation
holds \be \label{nc5}
<Y^{0},\mathbf{S}_{T}\widetilde{Y}^{T}>_{\mathcal{H}}=-\sigma_{2}(1)\int_{0}^{T}h(t)\widetilde{v}_{x}(1,t)dt,
\ee for any $\widetilde{Y}^{T}\in\mathcal{H}$, where $\widetilde{Y}$
is the solution to the adjoint problem \eqref{nc3}.
\end{Lem}
\begin{Proof}
The proof is similar to that of Lemma \ref{L1}
\end{Proof}\\
Following the ideas of Proposition \ref{xx}, we can prove the
following result.
\begin{Prop}\label{propos2} Problem (\ref{a})-(\ref{c}) is null-controllable in time $T > 0$ if and
only if for any $Y^{0} \in \mathcal{H}$, with Fourier expansion
\bea Y^{0}(x)&=&\sum_{n\in \mathbb{N}^{*}}Y_{n}^{0}\widetilde{\phi}_{n}(x),\eea there
exists a function $w(t)\in L^{2}(0,T)$ such that for all
$n\in\mathbb{N}^{*}$ \be-\frac{Y_{n}^{0}
\|\widetilde{\phi}_{n}\|^{2}}{\sigma_{2}(1)\phi_{n}^{v}(1)}e^{-\la_{n}T}=
\int_{0}^{T}w(t)e^{-\la_{n}t}dt.\label{nc6}\ee The control is given
by $w(t):=h(T-t)$.
\end{Prop}
\begin{proof} {\bf of Theorem \ref{nc1}.}
From the asymptotics \eqref{pron1} and \eqref{pron2}, it is easy to
show that for large $n\in\mathbb{N}^{*}$ there exists a constant $\mathbf{c}>0$ such that  \bean
\|\widetilde{\phi}_{n}(x)\|^{2}=\|\phi_{n}^{u}(x)\|_{L^{2}[-1,0]}^{2}+
\|\phi_{n}^{v}(x)\|_{L^{2}[0,1]}^{2}+
M|\phi_{n}^{v}(0)|^{2} \leq \mathbf{c}.\label{cccnn1}\eean
From the initial conditions in \eqref{w111}, the
expressions \eqref{wn}, \eqref{ggwn} and  the asymptotics \eqref{gnw21}, we have
\bean\label{pron3} \left\{
\begin{array}{ll}
\phi_{n}^{v}(1)=\left({\rho_{1}(0)\sigma_{1}(0)}\right)^{\frac{1}{4}}
\cos(\sqrt{\nu_{n}}\gamma_{1})+\mathcal{O}(\frac{1}{n}),&n\in\Lambda,\\
\phi_{n}^{v}(1)=(\rho_{1}(0)\sigma_{1}(0))^{-\frac{1}{4}}\sin(\sqrt{\nu_{n}}\gamma_{1})+\mathcal{O}(\frac{1}{n}),&n \in \mathbb{N}^{*}\backslash\Lambda.
\end{array}
 \right.
 \eean Since
$\phi_{n}^{v}(1)\neq0$ for all $n\in \mathbb{N}^{*}$, then
by use of \eqref{cccnn1}, \eqref{pron3} and the asymptotic \eqref{gn30}, we deduce that for
large $n\in\mathbb{N}^{*}$ there exists a constants $C_1>0$ and $C_2>0$ such that
$$\Big{|}\frac{Y_{n}^{0} \|\phi_{n}\|^2}{\sigma_{2}(1)\phi_{n}^{v}(1)}\Big{|}e^{-\la_{n}T}\leq
C_1 e^{-C_2n^2}.$$
By \eqref{gn30}, we have
\bean\sum_{n\in \mathbb{N}^{*}}\frac{1}{\nu_{n}}<\infty.\label{cdnn4}\eean
From \eqref{cdnn4} and Theorem \ref{Tn1}, there
exist a biorthogonal sequence $(\theta_{n}(t))_{n\in
\mathbb{N}^{*}}$ to the family of exponential functions
$(e^{-\la_{n}T})_{n\in \mathbb{N}^{*}}$.
On the other hand,
by \eqref{gn30} together with the
general theory developed in \cite{DH}, there exists constants $M>0$ and $\omega>0$
such that
$$\parallel\theta_{m}(t)\parallel_{L^{2}(0,T)}\leq M
e^{\omega m},~m \in \mathbb{N}^{*},$$ witch
implies the convergence of the series $$w(t)= \sum_{n\in
\mathbb{N}^{*}}-\frac{Y_{n}^{0}
\|\phi_{n}\|^{2}}{\sigma_{2}(1)\phi_{n}^{v}(1)}e^{-\la_{n}T}\theta_{n}(t).$$
Therefore, Theorem \ref{nc1} is a direct consequence of
the convergence of the series $w(t)$ in $L^2(0,T)$ and Proposition \ref{propos2}.
\end{proof}

\vskip 3cm

\end{document}